\def\ocirc#1{\ifmmode\setbox0=\hbox{$#1$}\dimen0=\ht0
    \advance\dimen0 by1pt\rlap{\hbox to\wd0{\hss\raise\dimen0
    \hbox{\hskip.2em$\scriptscriptstyle\circ$}\hss}}#1\else
    {\accent"17 #1}\fi}

\documentclass[a4paper]{amsart}
\usepackage{amsfonts}
\usepackage{amssymb}
\usepackage{amsmath}
\usepackage{newlfont}
\usepackage[mathcal]{euscript}
\usepackage{mathrsfs}
\usepackage{txfonts}
\usepackage{graphicx}
\usepackage{color}
\usepackage{enumerate}

 \newtheorem{question}{Question}
 \newtheorem{thm}{Theorem}

\newtheorem{lem}[thm]{Lemma}

\newtheorem*{tmrt}{Topological Multiple Recurrence Theorem}

\DeclareMathOperator{\Sp}{Sp}

 \newcommand{\ra}{\to}

 \newcommand{\N}{\mathbb{N}}

 \newcommand{\set}[1]{\left\{#1\right\}}

 \newcommand{\lang}{\mathcal{L}}

\newcommand{\fm}{f^{(\times m)}}
\newcommand{\fsm}{f^{(\ast m)}}
\newcommand{\Xm}{X^{m}}

\begin{document}


\title{On weak mixing, minimality and weak disjointness of all iterates}

\author[Kwietniak]{Dominik Kwietniak}
\address[D.~Kwietniak]{Institute of Mathematics, Jagiellonian University,
{\L}ojasiewicza~6, 30-348 Krak\'ow, Poland}
\email{dominik.kwietniak@uj.edu.pl}

\author{Piotr Oprocha}
\address[P.~Oprocha]{AGH University of Science and Technology\\
Faculty of Applied Mathematics\\
al. A. Mickiewicza 30, 30-059 Krak\'ow,
Poland\\ -- and --\\Institute of Mathematics\\ Polish Academy of Sciences\\ ul. \'Sniadeckich 8, 00-956 Warszawa, Poland} \email{oprocha@agh.edu.pl}

\begin{abstract}
The article addresses some open questions about the relations between the topological weak mixing property and
the transitivity of the map $f\times f^2 \times\ldots \times f^m$, where $f\colon X\ra X$ is a topological dynamical system on a compact metric space. The theorem stating that a weakly mixing and strongly transitive system is $\Delta$-transitive is extended to a non-invertible case with a simple proof. Two examples are constructed, answering the questions posed by Moothathu  [Colloq. Math. 120 (2010), no. 1, 127--138]. The first one is a multi-transitive non weakly mixing system, and the second one is a weakly mixing non multi-transitive system.
The examples are special spacing shifts. The later shows that the assumption of minimality in the Multiple Recurrence Theorem can not be replaced by weak mixing.
\end{abstract}



\maketitle


\section{Introduction}

The systematic study of transitivity and recurrence in dynamics dates back (as it is often the case in this subject) to Poincar\'{e}. In 1967 Furstenberg \cite{FurstenbergThm} published his seminal paper, which in recent years became the basis for a broad classification of dynamical systems by their recurrence properties. 
For an account of these results and their connections with combinatorics, harmonic analysis and number theory we refer the reader to Glasner survey article \cite{GlasnerRec}.

Our purpose here is to study recurrence properties of $f\times f^2\times \ldots\times f^m$. We clarify dependencies between some variants of transitivity by solving open problems posed by Moothathu \cite{TKSM}. Our interest in recurrence properties  of $f\times f^2\times \ldots\times f^m$ is motivated by the following version of the celebrated topological multiple recurrence theorem.  From it one can deduce the famous van der Waerden Theorem on the existence of arbitrarily long arithmetical progressions in some element of a partition of the integers (see \cite{GlasnerBook} pp. 46--47).

\begin{tmrt}{ \cite[Thm. 1.56]{GlasnerBook}}
Let $f$ be a minimal homeomorphism of a compact metric space $X$. If $U$ is a non-empty open subset of $X$, then for every positive integer $n$ there exists a positive integer $k$ with:
\[
U\cap f^k(U)\cap f^{2k}(U)\cap\ldots\cap f^{(n-1)k}(U)\neq\emptyset.
\]
\end{tmrt}

It follows that, if $f$ is a minimal homeomorphism, then for every $m\ge 1$ the map $f\times f^2\times \ldots f^m$ has a residual set of recurrent points. The last observation raises some natural questions: What other recurrence properties does $f\times f^2\times \ldots f^m$ have? Can it be minimal?  Must it be at least topologically transitive? Can we replace the assumption of minimality of $f$ by some other recurrence assumption like weak mixing?
We discuss some of these problems in Section \ref{sec:5}.
Here note that it is an immediate consequence of the above theorem that for every $n$ the set $N(U,U;f)=\set{m>0\; : \; f^m(U)\cap U\neq \emptyset}$ contains an arithmetic progression $k,2k,\ldots, k(n-1)$. Moreover, the same must hold if $f$ is continuous and topologically mixing. Then one can wonder if weak mixing is also enough. Since weak mixing implies that $N(U,U;f)$ contains arbitrary long intervals of consecutive integers, it is easy to see that in a weak mixing system for any non-empty open subset of $U\subset X$ and every positive integer $n$ there exist positive integers $k,m$ with:
\[
m+k, m+2k, \ldots, m+(n-1)k \in N(U,U;f).
\]
Now the question is: can we demand that $m=0$?
Our Theorem~\ref{thm:WMnotMT} shows that the answer must be in the negative.

Another formulation and motivation comes from the notion of \emph{disjointness} introduced to the topological dynamics, as well as to the ergodic theory by Furstenberg in  \cite{FurstenbergThm} and its \emph{weak} form developed in \cite{Auslander,GlasnerBook,HuangYe1,HuangYe2}. Let us recall, that $f$ and $g$ are \emph{weakly disjoint} if their Cartesian product $f\times g$ is topologically transitive. Weakly disjoint systems are kind of \emph{independent} one from another. It is independence in a rather weak sense as it may happen that $f$ is weakly disjoint from itself, that is, $f$ is \emph{weakly mixing}.
It is well known that $f$ is weakly mixing if and only if for any $n\ge 2$ the Cartesian product of $n$ copies of $f$, that is, $f\times\ldots\times f$ is topologically transitive. It follows that if $f$ is weakly mixing, then $f^n$ is topologically transitive for any $n\ge 1$.

Now, it is natural to ask: \emph{Can $f$ be weakly disjoint from some of its iterates, $f^m$, where $m\ge 2$?}
and \emph{How is weak disjointness of $f$ and $f^m$ related to weak mixing?}
These questions 
can be thought of as a topological dynamics counterpart of problems considered in ergodic theory (see \cite{Glasner}). Here we follow \cite{TKSM}, and we consider two properties, very similar to the weak mixing, namely:
\begin{description}
\item[$\mathbf{(\star)}$] for each $m\in \mathbb{N}$ the map $f\times f^2 \times\ldots \times f^m$ is topologically transitive.
\item[$\mathbf{(\star\star)}$] for each $m\in \mathbb{N}$ there is a residual set $Y\subset X$ such that for every point $x\in Y$ the tuple $(x,\ldots,x)\in \Xm$
has a dense orbit in $\Xm$ under the map $f\times f^2 \times\ldots \times f^m$.
\end{description}
Following \cite{TKSM}, we will say that $f$ is \emph{multi-transitive} if it satisfies $\mathbf{(\star)}$ and that $f$ is \emph{$\Delta$-transitive} if $\mathbf{(\star\star)}$ holds.

It is known that both properties presented above are equivalent to weak mixing if $f$ is a minimal homeomorphism. The proof of that equivalence using only elementary notions of topological dynamics is contained in \cite{TKSM}. The implication stating that weak mixing implies $\Delta$-transitivity
was earlier proved by Glasner (see \cite{Glasner}) with the help of the general structure theorem for minimal homeomorphisms. In \cite{TKSM} the question whether this implication holds for non necessarily invertible continuous maps was left open. Here we answer it affirmatively providing a simple proof for the general case, see Theorem~\ref{MT:thm4:general} below.

Moreover, we solve another open problem stated in \cite{TKSM}. We show that in general there is no connection between weak mixing and multi-transitivity by constructing examples of weakly mixing but non multi-transitive (Theorem \ref{thm:WMnotMT}) and multi-transitive but non weakly mixing (Theorem \ref{ex:MTnotWM}) systems. Finally, in Section \ref{sec:5} we offer some remarks regarding the last question of \cite{TKSM} in which Moothathu asked if there is a nontrivial minimal system $f\colon X\ra X$ such that $f\times f^2 \times\ldots \times f^m\colon \Xm\ra \Xm$ is minimal for some $m\ge 2$.

\section{Preliminaries}

Let $X$ be a compact metric space and $f\colon X\ra X$ be a continuous map. For every $m\geq 1$ denote the Cartesian product of $m$ copies of $X$ with itself by $\Xm$ and define two maps of $\Xm$ to itself: $\fm=f\times\ldots \times f$ and $\fsm=f\times f^2\times \ldots \times f^m$.

Given any sets $U,V\subset X$ we denote $N(U,V;f)=\set{n>0\; : \; f^{n}(U)\cap V \neq \emptyset}$.
If the map $f$ is clear from the context we simply write $N(U,V)$.

A map $f$ is \emph{minimal}, if it has no proper closed invariant set, that is, if $K\subset X$ {\color{green}is} nonempty, closed and $f(K)\subset K$ then $K=X$.
We say that $f$ is \emph{(topologically) transitive} if $N(U,V)\neq \emptyset$ for any pair of nonempty open sets $U,V\subset X$.
A set $S\subset\mathbb{Z}_+$ is \emph{syndetic} if there is a constant $L>0$ such that for every $n\ge$ we have $[n,n+L]\cap S\neq\emptyset$.
Then we say that a map $f$ is \emph{syndetically transitive} if $N(U,V)$ is syndetic for any nonempty open sets $U,V\subset X$. If $f\times f$ is transitive, then we say that $f$ is \emph{weakly mixing}. If for any nonempty open set $U\subset X$ there is $M>0$ such that $\bigcup_{j=1}^M f^j(U)=X$ then $f$ is said to be \emph{strongly transitive}. It immediately follows from the definition that any strongly transitive map is syndetically transitive.


Let $f$ and $g$ be two continuous surjective maps acting on compact metric spaces $X$ and $Y$, respectively.
We say that a nonempty closed set $J\subset X\times Y$  is \emph{a joining} of $f$ and $g$ if it is invariant for the product map $f\times g$ and its projections on first and second coordinate are $X$ and $Y$ respectively.
If $X\times Y$ is the only joining of $f$ and $g$ then we say that $f$ and $g$ are \emph{disjoint}.

The notion of disjointness was first introduced by Furstenberg in \cite{FurstenbergThm}.
It is well known that if $f$ and $g$ are disjoint then at least one of them is minimal. It is also not so hard to verify that if $f,g$ are both minimal, then they are disjoint if and only if $f\times g$ is minimal.

\section{Strong transitivity and $\Delta$-transitivity}

The main result of this section (Theorem~\ref{thm:main}) is obtained as a corollary from
Theorem~\ref{MT:thm4:general} below.
The  Theorem~\ref{MT:thm4:general} was proved by \cite[Theorem~4]{TKSM}
with the additional assumption that $f$ is a homeomorphism.
Here we present it with a new proof, which works for any continuous map.

We recall two results from \cite{TKSM}, modifying first to a suitable form.

\begin{thm}[{\cite[Proposition 1]{TKSM}}]\label{MT:prop1}
Let $X$ be a compact metric space. A continuous map $f\colon X\ra X$ is $\Delta$-transitive if and only if
for each $m\ge 1$ and nonempty open sets $U,V_1,\ldots,V_m\subset X$, there exists $n\ge 1$ such that
\[
U\cap \bigcap_{i=1}^m f^{-in}(V_i)\neq\emptyset.
\]
\end{thm}

\begin{thm}[{\cite[Corollary~2]{TKSM}}]\label{MT:cor2}
Let $X$ be a compact metric space. If $f\colon X\ra X$ is a weakly mixing and syndetically transitive continuous map,
then $f^{(*m)}$ is also weakly mixing and syndetically transitive
for any $m\geq 1$. In particular, $f$ is multi-transitive.
\end{thm}

The induction step in a proof of Theorem~\ref{MT:thm4:general} is based on the following:

\begin{lem}\label{lem:thiner_struct_mt}
Let $X$ be a compact metric space.
If $f\colon X\ra X$ is multi-transitive continuous map, then for
any $m\ge 1$ and nonempty open sets $V_1,\ldots, V_m\subset X$ there
is a sequence of integers $\{k_n\}_{n=0}^\infty$ such that for each $n\ge 0$ we have $k_n-n>0$ and for each $i=1,\ldots,m$ there is a sequence $\{V_i^{(n)}\}_{n=0}^\infty$ of nonempty open subsets of $V_i$  such that
\[
f^{i k_j-j}(V_i^{(n)})\subset V_i
\]
for $i=1,\ldots,m$, and $j=0,\ldots, n$.
\end{lem}
\begin{proof}
Let $V_1,\ldots,V_m$ be nonempty open subsets of $X$. Set $W=V_1\times\ldots \times V_m$.
We proceed by induction on $n$. From multi-transitivity of $f$
there is $k_0>0$ such that $(\fsm)^{k_0}(W)\cap W\neq \emptyset$, or equivalently
$f^{-ik_0}(V_i)\cap V_i\neq \emptyset$ for $i=1,\ldots,m$.
Put $V_i^{(0)}=f^{-ik_0}(V_i)\cap V_i\subset V_i$ for $i=1,\ldots,m$, to complete the base step.

For the induction step, suppose that $n\ge 1$ and we have found a sequence $k_0,\ldots,k_{n-1}$ and for each $i=1,\ldots,m$
we have nonempty open set $V_i^{(n-1)}\subset V_i$ such that
\begin{equation}\label{con:star}
f^{i k_j-j} (V_i^{(n-1)})\subset V_i\qquad\text{and}\qquad k_j-j>0, 
\end{equation}
hold for $j=0,\ldots,n-1$.
For $i=1,\ldots,m$, let $U_i=f^{-n}(V_i^{(n-1)})$. Put $U=U_1\times \ldots \times U_m$. By multi-transitivity
we get an integer $k_{n}$ such that $k_n-n>0$ and $(\fsm)^{k_n}(U)\cap W\neq \emptyset$, or equivalently
$f^{-ik_{n}}(V_i)\cap U_i\neq \emptyset$, for $i=1,\ldots,m$.
Fix $1\le i\le m$. We have
\[
f^{ik_n}(U_i)\cap V_i= f^{ik_n}(f^{-n}(V_i^{(n-1)}))\cap V_i=f^{ik_n-n}(V_i^{(n-1)})\cap V_i.
\]
By the above, $V_i^{(n)}=V_i^{(n-1)}\cap f^{-ik_n+n}(V_i)$ is nonempty,
open, and clearly $f^{ik_n-n}(V_i^{(n)})\subset V_i$. Moreover,
$V_i^{(n)}\subset V_i^{(n-1)}$. Using \eqref{con:star}, we conclude that
\[
f^{i k_j-j} (V_i^{(n)})\subset V_i
\]
for $j=0,\ldots,n$. This completes the proof.
\end{proof}



\begin{thm}\label{MT:thm4:general}
Let $X$ be a compact metric space. If $f\colon X\ra X$ is a weakly mixing and strongly transitive continuous map, then $f$ is $\Delta$-transitive.
\end{thm}
\begin{proof}
First, note that $f$ is multi-transitive by Theorem~\ref{MT:cor2}.
In particular, it is transitive, and surjective.

To prove that $f$ is $\Delta$-transitive we are going to use the equivalent condition provided by Theorem~\ref{MT:prop1}.
We will prove by induction on $m$ that for any nonempty open sets $U,V_1,\ldots,V_{m}\subset X$,
there exists $n\ge 1$ such that
\[
U\cap \bigcap_{i=1}^m f^{-in}(V_i)\neq\emptyset.
\]
For $m=1$ this statement simply follows from transitivity of $f$. Assume that we established the result for some $m\ge 1$.
We fix nonempty open sets $U$ and $V_1,\ldots,V_{m+1}$, and we want to show that
there are $n>0$ and $z\in U$ such that $f^{in}(z)\in V_i$ for $i=1,\ldots, m+1$.
By strong transitivity, $\bigcup_{j=1}^N f^j(U)=X$ for some $N>0$. Lemma~\ref{lem:thiner_struct_mt}  gives us
nonempty open sets $V_1^{(N)},\ldots,V^{(N)}_{m+1}$ and integers $k_0,\ldots,k_N$ such that
\[
f^{i k_l-l} (V_i^{(N)})\subset V_i \qquad\text{and}\qquad k_l>l,
\]
for $i=1,\ldots,m+1$ and $l=0,\ldots,N$. By the induction hypothesis we can find
$x\in V_1^{(N)}$ and $n>0$ such that $f^{in}(x)\in V_{i+1}^{(N)}$ for $i=1,\ldots,m$.
Clearly, there is $y\in X$ such that $f^n(y)=x$.
But strong transitivity gives us $f^j(z)=y$ for some $z\in U$ and $0\leq j \leq N$.
From the above we get
\begin{gather*}
f^{i(n+k_j)}(z)= f^{i(n+k_j)-j}(y)         =   f^{ik_j-j}(f^{in}(y)) =\\
               = f^{ik_j-j}(f^{(i-1)n}(x)) \in f^{ik_j-j}(V_i^{(N)})   \subset V_i
\end{gather*}
for any $i=1,2,\ldots,m+1$. We showed that
\[
z\in U \cap f^{-s}(V_1)\cap\ldots\cap f^{-s\cdot (m+1)}(V_{m+1}),
\]
where $s=n+k_j$, which completes the proof.
\end{proof}

\begin{thm}\label{thm:main}
Let $X$ be a compact metric space. If $f\colon X\ra X$ is a weakly mixing and minimal continuous map, then $f$ is $\Delta$-transitive.
\end{thm}
\begin{proof}
It is well known that any minimal map (invertible or not) on a compact metric space is strongly transitive (see \cite[Theorem 2.5(8)]{KST} for a proof).
We apply Theorem~\ref{MT:thm4:general} to finish the proof.
\end{proof}

Now we may formulate a general version of \cite[Corollary 7]{TKSM}, which was stated there for homeomorphisms.
Only the implication given by Theorem~\ref{thm:main} is new here. The rest of the proof is identical as in \cite{TKSM}.

\begin{thm}\label{thm:MMT}
Let $f\colon X\ra X$ be a minimal continuous map on a compact metric space $X$.
Then the following are equivalent:
\begin{enumerate}
\item $f\times f^2$ is transitive.
\item $f$ is multi-transitive.
\item $f$ is weakly mixing.
\item $f$ is $\Delta$-mixing.
\end{enumerate}
\end{thm}

\section{Weak mixing and multi-transitivity}

In \cite[page 10]{TKSM} T.~K.~S.~Moothathu asked the following question
\begin{question}
Are there any implications between weak mixing and multi-transitivity?
\end{question}

The aim of this section is to show that these notions are not related in a general situation, that is a continuous map can be multi-transitive and not weakly mixing, or weakly mixing and not multi-transitive. As it is often the case, to finish our task we will construct a symbolic systems.

Consider the set $A=\set{0,1}$ endowed with the discrete topology. Let $\Sigma$ denote the set of all infinite sequences of $0$'s and $1$'s regarded as the product of infinitely many copies of $A$ with the product topology. All sequences $x\in\Sigma$ are indexed by nonnegative integers, $x=x_0x_1x_2\ldots$. Then the \emph{shift} transformation is a continuous map $\sigma\colon\Sigma\ra\Sigma$ given by $\sigma(x)=y$, where $x=(x_i)$, $y=(y_i)$, and $y_i=x_{i+1}$ for $i=0,1,\ldots$. Any closed subset $X\subset\Sigma$ invariant for $\sigma$ is called a \emph{subshift} of $\Sigma$. A \emph{word} is a finite sequence of elements of $\set{0,1}$. The \emph{length} of a word $w$ is just the number of elements of $w$, and is denoted $|w|$. We say that a word $w=w_1w_2\ldots w_l$ appears in $x=(x_i)\in\Sigma$ at position $t$ if $x_{t+j-1}=w_j$ for $j=1,\ldots,l$. If $X$ is a subshift, then the \emph{language} of $X$ is the set $\mathcal{L}(X)$ of all words which appear at some position in some element $x\in X$. For any word $w$ let $[w]_t$ denote the element of the sequence $w$ standing at position $t$ and let $\Sp(w)=\set{|i-j|\;:\; [w]_i=[w]_j=1, i\neq j }$.
The set $\mathcal{L}_n(X)$ consists of all elements of $\mathcal{L}(X)$ of length $n$.

Let $P$ be a set of nonnegative integers. We say that a word $w=w_1w_2\ldots w_l$ is \emph{$P$-admissible} if $w_i=w_j=1$ for some $1\le i < j\le l$ implies $|i-j|\in P$, equivalently, if $\Sp(w)\subset P$.
Let $\Sigma_P$ be the subset of $\Sigma$ consisting of all sequences $x$ such that every word which appears in $x$ is $P$-admissible. It is easy to see that $\Sigma_P$ is a subshift, and $\mathcal{L}(\Sigma_P)$ is the set of all $P$-admissible words. We will write $\sigma_P$ for $\sigma$ restricted to $\Sigma_P$, and call the dynamical system given by $\sigma_P\colon\Sigma_P\ra\Sigma_P$ a \emph{spacing} shift. The class of spacing shifts was introduced by Lau and Zame in \cite{LZ}, and for a detailed exposition of their properties we refer to \cite{Spacing}.

Let $w$ be a $P$-admissible word. By $[w]_P$ we denote the set of all $x\in\Sigma_P$ such that the word $w$ appears at position $0$ in $x$. We call the set $[w]_P$ a \emph{$P$-admissible cylinder} (a \emph{cylinder} for short). The family of $P$-admissible cylinders is a base of topology of $\Sigma_P$ inherited from $\Sigma$. It is easy to see that definition of a spacing shift implies that $N([1]_P,[1]_P;\sigma_P)=P$. Moreover, $\sigma_P$ is weakly mixing if and only if $P$ is a \emph{thick} set (see \cite{LZ,Spacing}). A thick set is a subset of integers that contains arbitrarily long intervals ($P$ is thick if and only if for every $n$, there is  some $k$ such that $\{k,k+1,\ldots,k+n-1\}\subset P$).
If $w$ is a word and $n\ge 1$ then by $w^n$ we denote a word which is a concatenation of $n$ copies of $w$. If $n=0$ then $w^n$ is the \emph{empty word}.

\subsection{Multi-transitive and not weakly mixing example}\label{sec:ex1}
The results of this section generalize construction of totally transitive not weakly mixing spacing shift
presented in \cite{Spacing}.

We say that a finite set $S\subset \N$ is \emph{$q$-dispersed}, where $q\ge 2$, if for every $a,b\in S\cup\{0\}$ such that $a\neq b$
we have $|a-b|\ge q$.

\begin{lem}\label{lem:MTtool-NEW}
Let $M,N$ be positive integers such that $M\ge 3$ and let  $A\subset \N$ be an $M$-dispersed finite set.
Then there exists an $M$-dispersed finite set $B$ containing $A$ and such that
for $k=\max A+1$ and any pair of sequences of words $u_1,\ldots,u_N$ and $v_1,\ldots,v_N$ from $\lang_k(\Sigma_B)$ there is $n\geq 0$ such that
\[
\sigma^{in}([u_i]_B)\cap [v_i]_B\neq \emptyset \quad \text{ for }i=1,\ldots,N.
\]
\end{lem}
\begin{proof}
Let $k=\max A +1$.
Let $m=|\lang_k(\Sigma_A)|^{2N}$ be the cardinality of the set of all $N$-element sequences of pairs of words from
$\lang_k(\Sigma_A)$.
We enumerate all members of this set as a list $W^{(1)}, \ldots , W^{(m)}$. Hence, each $W^{(j)}$ is an ordered list of $N$ pairs of words from
$\lang_k(\Sigma_A)$:
\[
W^{(j)}=\left((u^{(j)}_1,v^{(j)}_1),\ldots,(u^{(j)}_N,v^{(j)}_N)\right),\qquad\text{for each}\;j=1,\ldots,m,
\]
where $(u^{(j)}_i,v^{(j)}_i)\in \lang_k(\Sigma_A)\times \lang_k(\Sigma_A)$ for every $i=1,\ldots,N$.
Choose integers $l_1,\ldots,l_m$ fulfilling the following conditions
\begin{align}
l_1
&\ge 2k+ M - 1,\label{cond-l1}\\
l_{j+1} &\ge (N+1)^{j} l_j.\label{cond-l2}
\end{align}
Given $1\le i \le N$ and $1\le j \le m$ we define
\[
w_i^{(j)}= u_i^{(j)}0^{i l_j-k}v_i^{(j)},
\]
where $l_1,\ldots,l_m$ are as above.
Using \eqref{cond-l1} and \eqref{cond-l2}, it is easy to see that
\begin{equation}\label{cond-l3}
[i l_\alpha-k+1,i l_\alpha+k-1]\cap [j l_\beta-k+1,j l_\beta+k-1]=\emptyset,
\end{equation}
for $1\le\alpha,\beta\le m$, $\alpha\neq\beta$ and $1\le i,j  \le  N$.
Let
\[
B=\bigcup_{j=1}^m \bigcup_{i=1}^N \Sp(w_i^{(j)}).
\]
If $n\in A$ then let $u=10^{n-1}10^{k-n-1}$. Clearly, $n\in\Sp(u)$ and $u\in \lang_k(\Sigma_A)$, since $k=\max A +1$. This gives $A\subset B$.
The construction of $w^{(j)}_i$ implies
that for $1\le i \le N$ and $1\le j \le m$ we have
\begin{equation}\label{cond-r}
\Sp(w^{(j)}_i)\setminus A\subset [il_j-k+1,il_j+k-1].
\end{equation}
Therefore,
\begin{equation}\label{cond-min}
\min B\setminus A\ge l_1 - k +1 \ge M+ k.
\end{equation}
In particular, $\min B=\min A\ge M$. Moreover,
we conclude form \eqref{cond-l3} and \eqref{cond-r} that if $r\in B\setminus A$, then there are unique indexes $i(r)$ and $j(r)$ such that $r\in\Sp(w^{(j(r))}_{i(r)})$.

Next, we are going to prove that $B$ is $M$-dispersed, that is, $|q-p|\ge M$ for each $q,p\in B$, $q\neq p$. We consider three cases:
\begin{description}
\item[Case I] \label{case-one} Both $p$ and $q$ belong to $A$. 
\item[Case II] \label{case-two} Both $p$ and $q$ belong to $B\setminus A$. 
\item[Case III] \label{case three} None of the above cases hold. 
\end{description}
The first case is clear, since $A$ is $M$-dispersed. The third case follows from \eqref{cond-min}. To prove the remaining case,
\textbf{Case II}, we consider subcases. But first note that in the computations below we use (\ref{cond-l1} - \ref{cond-r})
without further reference.
Given $p,q\in B\setminus A$ consider:
\begin{description}
\item[Case IIA] \label{case-2A} ${j(p)\neq j(q)}$. Without lost of generality we assume $j(q)>j(p)$. We have
\begin{eqnarray*}
q &\ge& i(q) l_{j(q)}-k+1 \ge  l_{j(q)}-k+1\\
&\ge& (N+1) l_{j(p)}-k+1\ge N l_{j(p)}-k+1 +l_1\\
&\ge& i(p) l_{j(p)}+k+M \ge p+M.
\end{eqnarray*}
But then
\[
q-p\ge M.
\]
\item[Case IIB] \label{case-2B} ${j(p)=j(q)}$, but ${i(p)\neq i(q)}$. Without lost of generality we assume $i(q)>i(p)$. Let $j=j(p)=j(q)$. Then
\begin{eqnarray*}
q &\ge& i(q) l_{j}-k+1 \ge (i(p)+1)\cdot l_{j}-k+1\\
&\ge& i(p) l_{j}-k+1 +l_1\ge i(p) l_{j}+k +M\ge p+M.
\end{eqnarray*}
Hence,
\[
q-p \ge M.
\]
\item[Case IIC] \label{case 2C} ${j(p)=j(q)}$, and ${i(p)= i(q)}$. Let $j=j(p)=j(q)$ and $i=i(p)= i(q)$.
For $r\in\{p,q\}$ we define
\[
s(r)=\min\left\{s: [w_{i}^{(j)}]_s=[w_{i}^{(j)}]_{s+r}=1\right\}.
\]
Clearly, either $s(p)\neq s(q)$, or $s(p)+p\neq s(q)+q$. We have
\begin{eqnarray*}
|q-p|&=&| (s(q)+q)-s(q) - (s(p)+p-s(p))|\\
&=&|(s(q)+q)-(s(p)+p) - (s(q)-s(p))|\\
&\ge& | |(s(q)+q)-(s(p)+p)| - |s(q)-s(p)||.
\end{eqnarray*}
But $|(s(q)+q)-(s(p)+p)|,|s(q)-s(p)|\in A\cup \{0\}$, so either
$|(s(q)+q)-(s(p)+p)|\neq |s(q)-s(p)|$ and then
\[|(s(q)+q)-(s(p)+p)| - |s(q)-s(p)|\ge M,\]
or $|(s(q)+q)-(s(p)+p)|=|s(q)-s(p)|\neq 0$, and then
\[
|q-p|\ge 2M.
\]
\end{description}

It remains to prove that for any pair of sequences of words $u_1,\ldots,u_N$ and $v_1,\ldots,v_N$ from $\mathcal{L}_k(\Sigma_B)$ there is $n\geq 0$ such that
\[
\sigma^{in}([u_i]_B)\cap [v_i]_B\neq \emptyset \quad \text{ for }i=1,\ldots,k.
\]
Observe that $\lang_k(\Sigma_B)= \lang_k(\Sigma_A)$, since $\min B\setminus A \ge k$, $\max A+1=k$, and $A\subset B$.
Therefore, according to our notation defined at the beginning of the proof,
for any two sequences of words $u_1,\ldots,u_N$ and $v_1,\ldots,v_N$ from $L_k(\Sigma_B)$, there is $j=1,\ldots,m$ such that
\[
 W^{(j)}=((u_1,v_1),\ldots,(u_N,v_N)).
\]
Let $w^{(j)}_i=u_i0^{il_j-k}v_i$ as above.
Clearly, $w^{(j)}_1,\ldots,w^{(j)}_N\in L(\Sigma_B)$, and from the definition of $w^{(j)}_i$ we conclude that
\[
\sigma^{in}\left(w^{(j)}_i\right)\in\sigma^{in}([u_i]_B)\cap [v_i]_B\qquad \text{for}\;n=l_j.
\]
Hence,
\[
\sigma^{in}([u_i]_B)\cap [v_i]_B\neq \emptyset \quad \text{ for }i=1,\ldots,N,
\]
where $n=l_j$.
\end{proof}

\begin{thm}\label{ex:MTnotWM}
There exists a set $P\subset\N$ such that the spacing shift $(\Sigma_P,\sigma_P)$ is multi-transitive but not weakly mixing.
\end{thm}
\begin{proof}
Fix any integer $M\ge 3$ and denote $P_0=\set{M}$. Define a sequence of sets $P_n\subset \N$ $(n\ge 1)$ inductively by putting  $P_{n+1}=B$, where $B$ is the set obtained for $A=P_n$, $N=n$, and $M$ as above by Lemma~\ref{lem:MTtool-NEW}. Denote
\[
P=\bigcup_{n=0}^\infty P_n.
\]
Easy induction gives  $|p-q|\geq M$ for every distinct $p,q\in P$ and  $P_0 \varsubsetneq P_1 \varsubsetneq P_2 \varsubsetneq \ldots$. In particular $P$ is not thick, so $\Sigma_P$ is not weakly mixing.
We are going to show that $\sigma_P\times \sigma_P^2 \times \ldots \times \sigma_P^m$ is transitive for
any $m=1,2,\ldots$. Fix any integer $m\ge 1$ and choose any open sets $U_1,\ldots, U_m, V_1,\ldots, V_m\subset\Sigma_P$.
Without loss of generality, we may assume that for each $1\le i \le m$ there are words $u_i,v_i\in \lang(\Sigma_P)$ such that $[u_i]_P\subset U_i$, and $[v_i]_P\subset V_i$.
We may also assume that for each $1\le i\le m$ we have $u_i,v_i\in \lang_k(\Sigma_{P_l})$ for some $l\ge m$ and $k=\max P_l+1$. The last equality implies that $\lang_k(\Sigma_{P_l})=\lang_k(\Sigma_P)$. If $m<l$ then we put $u_j=v_j=u_m$ for $j=m+1,\ldots,l$.

Now, by Lemma~\ref{lem:MTtool-NEW}, there is $j>0$  such that
\begin{eqnarray*}
\sigma_P^{ij}(U_i)\cap V_i &\supset & \sigma^{ij}([u_i]_P)\cap [v_i]_P\\
&\supset &\sigma^{ij}([u_i]_{P_l})\cap [v_i]_{P_l}\neq \emptyset
\end{eqnarray*}
for $i=1,\ldots,l$. We have just proved that $\sigma_P\times \sigma_P^2 \times \ldots \times \sigma_P^m$ is transitive for
any $m=1,2,\ldots$, which in other words means that $\sigma_P$ is multi-transitive.
\end{proof}

It is clear from the construction of $P$ in Lemma~\ref{lem:MTtool-NEW}, that the spacing shift $\sigma_P$ from the assertion of Theorem~\ref{ex:MTnotWM} is not syndetically transitive, since the set $P$, and as a result $N([1]_P,[1]_P)$, have thick complement. Then the following question arises:

\begin{question}
Does every multi-transitive and syndetically transitive system have to be weakly mixing?
\end{question}

\subsection{Weakly mixing and not multi-transitive example}\label{sec:ex2}

Fix $m\ge 2$. Let
\[
B(m,k)=\{m^{2k-1},m^{2k-1}+1,\ldots,m^{2k}-1\}, \qquad\text{and}\qquad P(m)=\bigcup_{k=1}^\infty B(m,k).
\]

Observe that for every $m\ge 2$ the set $P(m)$ has the following property
\begin{equation}\label{eq:not_mp}
p\in P(m) \implies m\cdot p\notin P(m).
\end{equation}

\begin{thm}\label{thm:WMnotMT}
Let $m\ge 2$ and $P=P(m)$ be as defined above. Then $\tau=\sigma_{P}\times\ldots \sigma_P^{m-1}$ is transitive, but $\tau\times\sigma_{P}^{m}$ is not transitive. In particular,
the spacing shift $(\Sigma_{P},\sigma_{P})$ is weakly mixing, but not multi-transitive.
\end{thm}

\begin{proof}It is easy to see that $P$ is thick, hence $\sigma_P$ is weakly mixing. To prove that
$\tau=\sigma_{P}\times\ldots \sigma_P^{m-1}$ is transitive, we fix open cylinders
\[
[u^{(1)}]_{P},\ldots, [u^{(m-1)}]_{P},[v^{(1)}]_{P},\ldots, [v^{(m-1)}]_{P}\in \mathcal{L}(\Sigma_P).
\]
Without lost of the generality we may assume that there is $k\ge 1$
such that for any $i=1,\ldots,m-1$ we have $|u^{(i)}|=|v^{(i)}|=t$, where $t=m^{2k}$.
Set $s=m^{2k+1}+m^{2k}$ and define
\[
w^{(i)}=u^{(i)}0^{is-t}v^{(i)},\qquad\text{where }i=1,\ldots,m-1.
\]
Clearly,
\[
[w^{(i)}]_{P}\subset \left(\sigma_{P}^{i}\right)^{-s}\left([v^{(i)}]_{P}\right)\cap [u^{(i)}]_{P},
\]
and therefore
\begin{eqnarray*}
&&[w^{(1)}]_{P}\times\ldots\times [w^{(m-1)}]_{P}\subset\\
&&\quad\quad\quad\tau^{-s}\left([v^1]_{P}\times\ldots\times [v^{(m-1)}]_{P}\right)\cap\left([u^{(1)}]_{P}\times\ldots\times [u^{(m-1)}]_{P}\right),
\end{eqnarray*}
so it is enough to prove that $[w^{(i)}]_{P}\neq\emptyset$, that is, $w^{(i)}\in\mathcal{L}(\Sigma_P)$.
It follows from definition of $w^{(i)}$ that
\[
\Sp(w^{(i)})=\Sp(u^{(i)})\cup \Sp(v^{(i)})\cup\{l-k:(l,k)\in \Delta\},
\]
where $\Delta$ is some subset of
\[\{0,\ldots,m^{2k}-1\}\times\{i\cdot m^{2k+1}+i\cdot m^{2k},\ldots,i\cdot m^{2k+1}+(i+1)\cdot m^{2k}-1\}.\]
Hence, we have
\[
l-k\in\{m^{2k+1},\ldots, m^{2k+2}-1\}\subset B(m,k+1),
\]
and $w^{(i)}\in \lang(\Sigma_P)$ as desired.
We proved that $\tau=\sigma_{P}\times\ldots \sigma_P^{m-1}$ is transitive.
To finish the proof it is enough to show that $\sigma_P\times\sigma^m_P$ is not transitive.
Let $U=V=[1]_P\times [1]_P$. It is easy to see from \eqref{eq:not_mp} that
\[
(\sigma_P\times\sigma^m_P)^n (U)\cap V = \emptyset
\]
for every $n\ge 0$, so $\sigma_P\times\sigma^m_P$ cannot be transitive.
\end{proof}

In the literature there are considered other recurrence properties stronger than weak mixing, see e.g. \cite{GlasnerRec}. It is natural to ask if we can replace weak mixing by one of them in Theorem~\ref{thm:WMnotMT}. In the view of the above results we would like to pose the following problem.

\begin{question}
Is there any nontrivial characterization of multi-transitive weakly mixing systems?
\end{question}

\section{Minimal self-joinings}\label{sec:5}

The last question in \cite{TKSM} asks: \emph{Can $f\times f^2 \times\ldots\times f^m\colon X^m\ra X^m$ be minimal if $m\ge 2$ and $X$ has at least two elements?} Let us call a map $f\colon X\ra X$ providing an affirmative answer to the above question \emph{multi-minimal}. Apparently, Moothathu posing his problem was not aware that the examples of multi-minimal homeomorphisms are known. But since their existence is stated in the language slightly different than terminology used in \cite{TKSM} we find it necessary to add some explanations. In fact the construction of multi-minimal systems is related to the considerations on multiple disjointness.

The first example of a system disjoint from any of its iterates (we are aware of) is the example of a POD (\emph{proximal orbit dense}) minimal homeomorphism given by Furstenberg, Keynes and Shapiro in \cite{FKS}. By Theorem 2.6 of \cite{Markley} every POD system has \emph{positive topological minimal self-joinings} (see \cite{Markley}). It also follows from Proposition 2.1 of \cite{Markley} that every homeomorphism possessing positive topological minimal self-joinings is multi-minimal, and so is the example from \cite{FKS}. Furthermore, del Junco's work \cite{delJunco}, together with his joint work with Rahe and Swanson \cite{dJRS} shows that Chacon's example \cite{Chacon} is POD, and hence also multi-minimal. In \cite{AM} Auslander and Markley introduced the class of \emph{graphic} minimal systems, which generalizes POD homeomorphisms. They also proved that each graphic flow is multi-minimal \cite[Corollary~22]{AM}. Moreover, as announced in \cite[page 490]{AM} Markley constructed an example of a graphic homeomorphisms which is not POD, hence it is another kind of multi-minimal homeomorphism.

More information about minimal subsystems of $f\times f^2 \times\ldots\times f^m$ is to be found in \cite{AM2,BGK,BK,delJunco2,King} to name only a few. There is also in some sense parallel and certainly deep theory of minimal self-joinings (a part of ergodic theory), introduced by Rudolph \cite{R}, see Glasner's book \cite{GlasnerBook}. We remark that although every weak mixing minimal map is multi-transitive it is not necessarily multi-minimal. The \emph{discrete horocycle flow} $h$ is an example of a weakly mixing minimal homeomorphism such that $h$ is topologically conjugated
to $h^2$, and hence it is not multi-minimal (see \cite[pages 26, and 105-110]{GlasnerBook}).
The facts gathered above prompt us to raise following questions:

\begin{question}
Is there any nontrivial characterization of multi-minimality in terms of some dynamical properties?
\end{question}

It is also interesting whether is it possible to characterize multi-minimal systems adding some mild assumptions to Theorem~\ref{thm:MMT}.
In particular, we don't know the answer for the following question.

\begin{question}
Assume that $f$ is a weakly mixing map such that $f\times f^2$ is minimal. Is $f$ necessarily multi-minimal?
\end{question}

\section{Acknowledgements}
We are grateful to the referee for a number of helpful suggestions for improvement in the article.
We would like to express many thanks to Professor Eli Glasner for his remarks on connections between disjointness and spectral type of minimal flows. They resulted in the survey exposition in Section~\ref{sec:5}.

This work was supported by the Polish Ministry of Science and Higher Education from sources for science in the years 2010-2011, grants  no. NN201270035 (D. Kwietniak) and IP2010~029570 (P. Oprocha).

A part of the present work was done during the visit of the first author at the University of Murcia. At the time the second autor was Marie Curie Fellow (European Community's Seventh Framework Programme FP7/2007-2013, grant no. 219212), supported additionally by MICINN and FEDER (grant MTM2008-03679/MTM).
The financial support and hospitality of these institutions is hereby gratefully acknowledged.

\end{document}